\documentclass[12pt]{amsart}
\usepackage[style=numeric,natbib=true]{biblatex}
\usepackage[utf8]{inputenc}
\usepackage[T1]{fontenc}
\usepackage{hyperref}
\usepackage{amsmath}
\usepackage{amsfonts}
\usepackage{amssymb}
\usepackage{amsthm}
\usepackage{mathtools}
\usepackage{latexsym}
\usepackage{euscript}
\usepackage{xspace}
\usepackage{ifthen}
\usepackage{xstring}
\usepackage{txfonts}
\usepackage{calrsfs}
\usepackage{cleveref}
\usepackage{cancel}

\newcommand\ORCiD[1]{\href{https://orcid.org/#1}{\mbox{\scalerel*{
\begin{tikzpicture}[yscale=-1,transform shape]
\pic{orcidlogo};
\end{tikzpicture}
}{|}}}}
\renewcommand\ORCiD[1]{}
\newcommand{\amorcid}{0000-0003-0661-4495}
{\vspace{1em}
\begin{minipage}{12cm}
\rule{12cm}{1pt}\sffamily\noindent}%
{\newline\rule{12cm}{1pt}
\end{minipage}
 \vspace{1em}}

\newtheorem{teo}{Theorem}
\newtheorem{lem}[teo]{Lemma}
\newtheorem{pro}[teo]{Proposition}

{\begin{enumerate}}%
{\end{enumerate}}

\newcommand{\N}{\ensuremath{\mathbb{N}}\xspace}

\DeclarePairedDelimiter\floor{\lfloor}{\rfloor}
\newcommand{\ccomma}{\mathbin{\raisebox{0.5ex}{,}}}
\crefname{cor}{Corollary}{Corollaries}

\newcounter{numex}

\newcommand{\ub}{\ensuremath{\varepsilon}\xspace}

\usepackage{graphicx}
\hypersetup{colorlinks,%
citecolor=cyan,%
linkcolor=red,%
pdftex}
\usepackage{color}

\newcommand{\palsym}{\ensuremath{\sigma\!}}
\newcommand{\oldpal}[1]{
  \StrCut{#1}{;}\palARG\palX
  \IfStrEq{\palX}{\empty}{\def\palX{x}}{}
  \ensuremath{
    \IfSubStr{\palARG}{,}{
        \palsym^{\left(\StrBefore{\palARG}{,}\right)}_{\StrBehind{\expandafter\palARG}{,}}\left(\palX\right)
      }
      {\palsym_{\palARG}\left(\palX\right)}
    }
  }

\newcommand{\eipi}[1]{\ensuremath{
    \IfSubStr{#1}{;}{
      e^\frac{2i\pi \StrBefore{#1}{;}}{\StrBehind{#1}{;}}}
    {e^{2i\pi#1}}
}}

\title[Large finite products]{Large finite products of small fractions}

\author[A.~Mandel]{Arnaldo Mandel} 
\thanks{This work was partially supported by Conselho Nacional de
Desenvolvimento Científico e Tecnológico – CNPq (Proc.
423833/2018-9).} 
  \address{Computer Science Department, Instituto de Matem\'atica e Estat\'\i stica, Universidade de
  S\~ao Paulo\\
  \emph{\small S\~ao Paulo, SP, Brazil 05508-970}\\
\textsc{Orcid}: {\ORCiD{\amorcid}}\amorcid}
\email{ am@ime.usp.br}

\begin{document}

\begin{abstract}
  Fix positive reals \(a,b,c,d\), and let \(h(x)\) be a real
  function behaving sort of like \(\sin x\) near 0.  Then,
  provided \(m\) grows linearly with \(n\). there exists a
  positive constant \(C\) such that\[
    \prod_{j=0}^m\frac{h\left((cj+a)\frac{d}{n}\right)}{h\left((cj+b)\frac{d}{n}\right)}\sim C n^{\frac{a-b}c}.
  \]  
\end{abstract}
\maketitle


Let \(n\geq 4\) be an integer, denote \(\theta=\frac\pi{2n}\),
and consider the finite product
\[
  D_n=\frac{\sin 5\theta}{\sin 3\theta}\cdot\frac{\sin 9\theta}{\sin 7\theta}\cdot\frac{\sin 13\theta}{\sin 11\theta}\cdots
\]  
where the terms go on while the arguments of \(\sin\) stay below
\(\frac\pi2\).  The proof of an important result in \citet{MR}
hinges on showing that \(D_n\) grows unboundedly with \(n\).  It
is shown there that \(D_n=\Omega(n^{\frac12-\varepsilon})\),
which is enough; here we remove the annoying \(-\varepsilon\)
from the exponent and determine the precise order of growth.

Here is an outrageous idea,  just do some obvious cancellations:\\
\(\prod_j\frac{\sin(4j+5)\theta}{\sin(4j+3)\theta}=\
\prod_j\frac{\cancel\sin(4j+5)\theta}{\cancel\sin(4j+3)\theta}=
\prod_j\frac{(4j+5)\cancel\theta}{(4j+3)\cancel\theta}=
\prod_j\frac{4j+5}{4j+3}\),
and proceed from there.

Some people will have an issue with that, of course; however,
this dumb idea turns out to be useful!  Indeed, we will tackle a
strong generalization of the product above, presenting good
asymptotics, and the product obtained by illegal cancellation
will be a major tool.

The \(\sin\) function is not very special in this context.  The following encapsulates
what about it is relevant here. We use primes to denote derivatives.
\begin{pro}\label{hcond}
  For a real function \(h\), analytic around \(0\), the following are equivalent:

  \begin{enumerate}
      \item \(h(0)=h''(0)=0\), \(h'(0)>0\), and \(h''(x)\leq 0\)
    for positive \(x\) close to \(0\).
      \item \(h(x)\) is the identity function or there exist
    reals \(\alpha,\lambda>0\) and integer \(k\geq 3\) such that
    \(h(x)=\alpha(x-\lambda x^k)+O(x^{k+1})\).
  \end{enumerate}
\end{pro}
\begin{proof}
  Exercise.
\end{proof}

For convenience, call a function as above an \emph{S-function} (a
generalized sine, so to speak).

Let \(a,b,c,d\) be positive reals; they are supposed to be
constant, throughout. We consider an additional positive real
parameter \ub, subject to the following
\emph{compatibility condition}: \(\ub\leq cd\), and, if
\(\ub=cd\), then \(cd>1\); this weird condition will only surface
in \cref{decrease}. Finally, given \(n\in\N\), let
\(m=m(n)\in\N\) be maximum such that both
\((cm+a)\frac{d}{n}\ccomma(cm+b)\frac{d}{n}\leq\ub\), that is
\(m=\floor{\frac{n\ub}{cd}-\frac{\max(a,b)}c}\).  Given a real
function \(h\), we define two products:
\begin{align*}
  D_n(a,b,c,d,\ub;h) &= \prod_{j=0}^m\frac{h\left((cj+a)\frac{d}{n}\right)}{h\left((cj+b)\frac{d}{n}\right)}\ccomma\\
  K_n(a,b,c,\ub)   &= \prod_{j=0}^m\frac{cj+a}{cj+b}\cdot
\end{align*}

Note that if one chooses \(h(x)=x\), then
\(K_n(a,b,c,\ub)=D_n(a,b,c,d,\ub;h)\), just a special case -- or,
as mentioned before, by silly cancellation of \(h\)'s on the
expression of \(D_n\).  Note that there is no loss of generality
in taking \(d=1\), but we keep the extra parameter to cater for
the looks of the motivating example.

\begin{teo}\label{asymp}
  Suppose that \(h\) is an S-function, and
  let \(H(x)=\frac{h(x)}x\). Assume that a compatible \ub is such
  that in \([0,\ub]\) we have that \(H(x)>0\) and
  \(H''(x)\leq 0\). Then, there exists a constant \(C=C(a,b,c,\ub;h)\)
  such that
  \[
    D_n(a,b,c,\ub;h) \sim Cn^{\frac{a-b}c}.
  \]  
\end{teo}

Notice that, given \(h\), one can always choose \ub as required,
since \(H(x)=1-\lambda x^{k-1}+O(x^k)\).

The motivating example in \cite{MR} is
\(D_n(5,3,4,\pi/2,\pi/2;\sin)\); some routine algebraic
manipulation show that \(\ub=\frac\pi2\) conforms to the
requirements of the Theorem, and we obtain \(D_n \sim C\sqrt{n}\)
for some constant \(C\).

The result will be obtained by comparing \(D_n\) and \(K_n\).  The
asymptotics for \(K_n\) is well known (it essentially appears in
\cite[\(11^{\text{th}}\) formula line]{diek}).

\begin{pro}\label{kprod}
  \(K_n(a,b,c,\ub)\sim
\frac{\Gamma(b/c)}{\Gamma(a/c)}\left(\frac{\ub}c\right)^{\frac{a-b}c}n^{\frac{a-b}c}\).
\end{pro}
\begin{proof}
  We can rewrite
  \[
    K_n(a,b,c,\ub)=\prod_{j=0}^m\frac{j+a/c}{j+b/c}=
    \frac{\Gamma(b/c)}{\Gamma(a/c)}\frac{\Gamma(m+1+a/c)}{\Gamma(m+1+b/c)}\cdot
  \]
  The last quotient is asymptotic to \(m^{\frac{a-b}c}\) (this
  follows easily from Stirling's formula; it falls into ``well
  known'', see \cite{wendel}, \cite[eq. 5.11.12]{nist}). The result follows by noticing that \(m\sim \frac{\ub n}c\).
\end{proof}

Without loss of generality, we will assume \(d=1\) from now on, and remove it
altogether from the notation.  Define
\begin{equation}
  \label{eq:En}
  E_n(a,b,c,\ub;h)=\frac{D_n(a,b,c,\ub;h)}{K_n(a,b,c,\ub)}=
  \prod_{j=0}^m\frac{H\left(\frac{cj+a}{n}\right)}{H\left(\frac{cj+b}{n}\right)}\cdot  
\end{equation}

Our goal will be met by showing that
\(\lim_{n\rightarrow\infty}E_n\) exists and is positive.  This
result is as interesting as \cref{asymp} itself, so we state it
in full, granting  \(H\) first class status, in parallel with \cref{hcond}.

\begin{pro}\label{Hcond}
  For a nonconstant real function \(H\), analytic around \(0\), the following are equivalent:
  \begin{enumerate}
      \item \(H(0)>0\),\(H'(0)=0\), and \(H''(x)\leq 0\) for
    positive \(x\) close to \(0\). 
      \item There exist
    reals \(\alpha, \lambda>0\) and integer \(k\geq 2\) such that
    \(H(x)=\alpha(1-\lambda x^k)+O(x^{k+1})\).
  \end{enumerate}
\end{pro}
Motivated as before, we call such a function a \emph{C-function}.
Clearly, \(h(x)\) is an S-function if and only if \(h(x)/x\)
is a C-function.  Therefore,
\(E_n(a,b,c,\ub;h)=D_n(a,b,c,\ub;H)\) if \(H(x)=h(x)/x\).


\begin{teo}\label{limexists}
  Suppose that \(H\) is a C-function.
  Assume that a compatible \ub is such
  that in \([0,\ub]\) we have that \(H(x)>0\) and
  \(H''(x)\leq 0\). Then \(\lim_{n\rightarrow\infty}D_n(a,b,c,\ub;H)\) exists and is positive.
\end{teo}

Noting that \(D_n(a,b,c,d,\ub;h)=D_n(b,a,c,d,\ub;h)^{-1}\),
we will assume, in what follows, that \(b<a\), as this will
entail both theorems in full.  The hypotheses of either Theorem
are assumed in the following lemmas, and \(h(x)=xH(x)\).

\begin{lem}\label{omega}
  For all sufficiently large \(n\), \(D_n(a,b,c,\ub;H)\) is
  bounded away from \(0\) -- that is,
  \(D_n(a,b,c,\ub;h)=\Omega(n^{\frac{a-b}c})\).
\end{lem}
\begin{proof}
  We will prove below that there exists a positive constant \(A\) (independent of \(n\))
  such that for all relevant \(j\), and sufficiently large \(n\),
  \begin{equation}
    \label{eq:low}
    \frac{H\left(\frac{cj+a}{n}\right)}{H\left(\frac{cj+b}{n}\right)}\geq 1-\frac{A}{m}\cdot
  \end{equation}
  Having proved that, it follows that
  \[
      D_n(a,b,c,\ub;H)\geq  \left(1-\frac{A}{m}\right)^{m+1}\ccomma
  \]
  and the right hand side converges to \(e^{-A}\), proving the Lemma.

  It remains to prove \labelcref{eq:low}.  For that matter, consider
  parameters \(\delta,\alpha>0\), and define
  \[
      g_{\delta}(x,y)=\frac1x\left(1-\frac{H\big((y+\delta)x\big)}{H(yx)}\right)\ccomma
  \]
  for \(x> 0, y\geq\alpha, yx\leq\ub\), and with
  \(g_{\delta}(0,y)=0\).  One easily verifies that \(g_\delta\)
  is continuous: from the expression, this is only an issue for
  \(x=0\), and that is quickly handled using the Taylor
  approximation for \(H\) (this is also where the requirement
  that \(k\geq 2\) in the definition of C-function shows its
  hand).  It follows that \(g_\delta\) attains a maximum
  \(A(\delta)\), hence, for all \(x,y\) in the domain,
  \[
     \frac{H\big((y+\delta)x\big)}{H(yx)}\geq 1-A(\delta)x.
  \]

  To obtain \labelcref{eq:low} we take \(x=\frac{1}{n}\ccomma\)
  \(\alpha=b\), \(y=cj+b\), \(\delta=a-b\).  We are almost done,
  except that the right hand side reads \(1-\frac{A(\delta)}n\).
  Since \(m=\floor{\frac{\ub n}c-a}\geq \frac{\ub n}c-a-1\),
  \(\frac{m}n\geq \frac{\ub}c-\frac{a+1}n\geq \frac{\ub}{2c}\) for sufficiently large \(n\).
  Choosing now \(A=A(\delta) \frac{\ub}{2c}\) yields  \labelcref{eq:low}.
  
\end{proof}

\begin{lem}\label{decrease}
  For all relevant \(j\) and sufficiently large \(n\),
  \[
    \frac{H\left(\frac{cj+a}{n}\right)}{H\left(\frac{cj+b}{n}\right)}\geq\frac{H\left(\frac{c(j+1)+a}{n+1}\right)}{H\left(\frac{c(j+1)+b}{n+1}\right)}\cdot
  \]
\end{lem}
\begin{proof}
  This is clearly equivalent to proving that
  \begin{equation}
    \label{eq:dec}
    \frac{H\left(\frac{cj+a}{n}\right)}{H\left(\frac{c(j+1)+a}{n+1}\right)}>
    \frac{H\left(\frac{cj+b}{n}\right)}{H\left(\frac{c(j+1)+b}{n+1}\right)}\cdot            
  \end{equation}

  Let
  \(\delta(t)=\frac{c(j+1)+t}{n+1}-\frac{cj+t}{n}=\frac{(n-j)c-t}{n(n+1)}\)
  and \(f(x)=\frac{H(x)}{H(x+\delta(a))}\); notice that
  \(n(n+1)\delta(a)\geq(n-m)c-a\geq n(c-\ub)+a(c-1)>0\) for
  all sufficiently large n, by compatibility.  Let us show that \(f\)
  is increasing; it is enough to show that its logarithmic
  derivative, that is
  \[
    \frac{d}{dx}\log
    f(x)=\frac{H'(x)}{H(x)}-\frac{H'(x+\delta(a))}{H(x+\delta(a))}\ccomma
  \]
   is positive.
  Since \(\delta(a)>0\), this  will follow if we show that
  \(\frac{H'(x)}{H(x)}\) is decreasing; that follows, as its derivative is
  \(\frac{HH''-H'^2}{H^2}\), which, by the choice of \ub, is
  negative.  So, we have proved that \(f\) is increasing.  Taking
  \(X=\frac{cj+a}{n},\,Y=\frac{cj+b}{n}\), we have that \(X>Y\),
  hence \(f(X)>f(Y)\), that is,
  \[
    \frac{H(X)}{H(X+\delta(a))}>
    \frac{H(Y)}{H(Y+\delta(a))}>
    \frac{H(Y)}{H(Y+\delta(b))}\ccomma
  \]
  where the last inequality follows since \(\delta(b)>\delta(a)\)
  and \(H\) is decreasing.  Expanding \(X\) and \(Y\), we obtain
  \labelcref{eq:dec}.
\end{proof}

\begin{proof}[Proof of \cref{asymp}]
  It follows from \cref{decrease} that for sufficiently large
  \(n\) \(E_{n+1}(a,b,c,\ub;h)\) is a product of terms smaller
  than the corresponding terms of \(E_n\), and a few more terms,
  all of them \(<1\).  So, the products \(E_n\) form a decreasing
  sequence, which by \cref{omega} is bounded above \(0\).
  Therefore, it has a positive limit \(C_0\), so
  \(D_n(a,b,c,\ub;h) \sim C_0 K_n(a,b,c,\ub) \sim
  Cn^{\frac{a-b}c}.\)
\end{proof}

It would be nice to describe \(C=C(a,b,c,\ub;h)\) in terms of the
parameters. Maybe a precise estimate of \(A\) in \cref{omega}
would cinch it.  We present an upper bound for \(C\), assuming \(a>b\):

\begin{pro}
  \(C(a,b,c,\ub;h)\leq
  \frac{\Gamma(b/c)}{\Gamma(a/c)}\left(\frac{\ub}c\right)^{\frac{a-b}c}.\)
\end{pro}
\begin{proof}
  Since \(H(x)\) is decreasing, \(E_n\) is a product of terms
  \(<1\), hence \(E_n<1\). The result follows from \cref{kprod}.
\end{proof}

There are several common functions in each of the two classes:
\(\sin\), \(\arctan\), \(\tanh\), \(\sinh^{(-1)}\), erf are
S-functions, while \(\cos\), \(\cot^{(-1)}\), \(\text{sech}\), \(e^{-x^2}\),
\((1+x^2)^{-1}\) are C-functions, and one can produce plenty of
rational functions on each class.  Besides, the C-functions form
a semiring with pointwise sum and product, and the S-functions
are a semimodule over that semiring; also the derivative of an
S-funtion is a C-function.  All together, one can write very
impressive products, provided one can come with a nicely
expressed \ub (for instance, if \(h(x)=\sin x\), one can take
\(\ub=\frac\pi2\)).  It may happen that for some suitable choice
of parameters, coincidence happens, a slick proof is forthcoming
and even an exact result can be provided.  That could become an
interesting exercise or competition question.  For instance:

\noindent\textsc{Exercise:} Prove that, for \(k\geq2\),
\(\lim_{n\rightarrow\infty}D_n\left(a,b,c,d,\left(\frac{k-1}k\right)^{\frac1k},e^{-x^k}\right)=e^{-\frac{k-1}k\frac{a-b}c}\).
Notice that \(e^{-x^k}\) is a C-function and the choice of \ub
was driven by the conditions of \cref{limexists}.  Never mind
about compatibility, which in this specific case is not
necessary.

We close with two open problems:

\noindent\textsc{Problem 1:} Find an insightful expression for
\(C\) as in \cref{asymp} or the limit in \cref{limexists} in
terms of the parameters.

\noindent\textsc{Problem 2:} Estimate the rate of convergence to
the limit in \cref{limexists}. A little computational
experimentation suggests that it is slow, the difference between
\(D_n\) and the limit behaving as \(O(1/\log n)\).

\printbibliography
\end{document}